\documentclass[12pt]{amsart}
\usepackage{amsfonts}
\usepackage{amssymb}
\usepackage{color}
 \newtheorem{theorem}{Theorem}[section]
 
 \newtheorem{conjecture}[theorem]{Conjecture}
 \newtheorem{lemma}[theorem]{Lemma}
 \newtheorem{proposition}[theorem]{Proposition}
 \newtheorem{problem}[theorem]{Problem}
 \theoremstyle{definition}
 
 \theoremstyle{remark}

\numberwithin{equation}{section}

\DeclareMathOperator{\ord}{ord}

\begin{document}

\title[On the number of subsequence sums]
{On the number of subsequence sums related to the support of a sequence in finite abelian groups}

\author[R. Wang]{Rui Wang}
\address{College of Science, Civil Aviation University of China,
Tianjin, P.R. China 300300}\email{r-wang@cauc.edu.cn}

\author[H. Chao]{Han Chao}
\address{College of Science, Civil Aviation University of China,
Tianjin, P.R. China 300300}\email{974835335@qq.com}

\author[J. Peng]{Jiangtao Peng*}
\address{College of Science, Civil Aviation University of China,
Tianjin, P.R. China 300300}\email{jtpeng1982@aliyun.com}

\keywords{Subsequence sums; $k$-sum; Support of sequence; Abelian group.}


\thanks{*Corresponding author}

\begin{abstract}
Let $G$ be a finite abelian group and $S$ a sequence with elements of $G$. Let $|S|$ denote the length of $S$ and $\mathrm{supp}(S)$ the set of all the distinct terms in $S$. For an integer $k$ with $k\in [1, |S|]$, let $\Sigma_{k}(S) \subset G$ denote the set of group elements which can be expressed as a sum of a subsequence of $S$ with length $k$. Let $\Sigma(S)=\cup_{k=1}^{|S|}\Sigma_{k}(S)$ and $\Sigma_{\geq k}(S)=\cup_{t=k}^{|S|}\Sigma_{t}(S)$. It is known that if $0\not\in \Sigma(S)$, then $|\Sigma(S)|\geq |S|+|\mathrm{supp}(S)|-1$. In this paper, we determine the structure of a sequence $S$ satisfying $0\notin \Sigma(S)$ and $|\Sigma(S)|= |S|+|\mathrm{supp}(S)|-1$. As a consequence, we can give a counterexample of a conjecture of Gao, Grynkiewicz, and Xia. Moreover, we prove that if $|S|>k$ and  $0\not\in \Sigma_{\geq k}(S)\cup \mathrm{supp}(S)$, then $|\Sigma_{\geq k}(S)|\geq |S|-k+|\mathrm{supp}(S)|$. Then we can give an alternative proof of a conjecture of Hamidoune, which was first proved by Gao, Grynkiewicz, and Xia.
\end{abstract}

\date{}

\maketitle

\section{Introduction and main results}

Let $G$ be a finite abelian group and  $S$ a finite sequence with elements of $G$. Let $|S|$ denote the {\it length} of $S$, $\sigma(S)$ the {\it sum} of the elements of $S$. Let $\mathrm{supp}(S)$ denote the {\it support} of $S$, which is defined as the set of all the distinct terms in $S$. For every integer $k\in [1, |S|]$, let
\begin{equation*}
\Sigma_{k}(S)=\{ \sigma(T) \mid T \mbox{ is a subsequence of } S  \mbox{ with }  |T|=k\}
\end{equation*}
denote the set of {\it $k$-sums} of $S$. Let
\begin{equation*}
\Sigma(S)=\bigcup_{k=1}^{|S|}\Sigma_{k}(S) \mbox{ and } \Sigma_{\geq \ell}(S)=\bigcup_{k=\ell}^{|S|}\Sigma_{k}(S).
\end{equation*}
where $\ell \in [1, |S|]$. A sequence $S$ is called {\it zero-sum }  if $\sigma(S) = 0 \in G$, and {\it zero-sum free}  if $0 \not\in \Sigma(S)$.

The subsequence sum problems study when do we have $0\in \Sigma(S)$ ($0\in \Sigma_{k}(S)$, respectively), and what is the lower bound of $|\Sigma(S)|$ ($|\Sigma_{k}(S)|$, respectively) if $0\not\in \Sigma(S)$ ($0\not\in \Sigma_{k}(S)$, respectively).

\medskip
For a finite abelian group $G$, let $\mathsf D(G)$ denote the smallest integer $\ell \in \mathbb{N}$ such that $0\in \Sigma(S)$ for every sequence $S$ over $G$ of length $\ell$.  We call $\mathsf D(G)$ the {\it Davenport constant} of $G$.  The problem of determining $\mathsf D(G)$ was first proposed by Rogers \cite{Rogers} in 1963, and attracted many researchers (see the survey \cite{GG2006}). An inverse problem associated with the Davenport constant asks that if $0\not\in  \Sigma(S)$, what is the lower bound of $|\Sigma(S)|$?

Let $S$ be a zero-sum free sequence over a finite abelian group $G$. It is not hard (see Lemma~\ref{lemma of ZSFS}) to prove that $|\Sigma(S)|\geq |S|$, and equality holds if and only if $\mathrm{supp}(S)=\{g\}$ for some $g\in G$ with $\ord(g)\geq |S|+1$. It suggests that $|\Sigma(S)|$ and $|\mathrm{supp}(S)|$ are not independent. In 1972, Eggleton and Erd\H{o}s \cite{EE72} proved that $|\Sigma(S)|\geq 2|\mathrm{supp}(S)|-1$. In 2001, Freeze, Smith \cite{FS01} proved the following result, which also can be found in \cite[Proposition 5.3.5]{GH06}.
\begin{theorem}\cite[Proposition 5.3.5]{GH06}\label{GHZSFS}
Let $G$ be a finite abelian group and $S$ a zero-sum free sequence over $G$. Then
\begin{equation*}
|\Sigma(S)|\geq |S|+|\mathrm{supp}(S)|-1.
\end{equation*}
\end{theorem}

By Lemma~ \ref{lemma of ZSFS}, the lower bound in Theorem~\ref{GHZSFS} is the best possible.

\medskip

The famous Erd\H{o}s-Ginzburg-Ziv Theorem \cite{EGZ}, proved in 1961, states that if $G$ is a cyclic group of order $n$ and $S$ is a sequence over $G$ with $|S|\geq 2n-1$, then $0\in \Sigma_{n}(S)$. In 1996, Gao \cite{Gao3} proved that if $G$ is an abelian group of order $n$ and $S$ is a sequence over $G$ with $|S|\geq n+\mathsf D(G)-1$, then $0\in \Sigma_{n}(S)$. This result, which gives a fundamental relation between the  Erd\H{o}s-Ginzburg-Ziv Theorem and the Davenport constant, is called the Gao Theorem.

Let $G$ be an abelian group of order $n$ and $S$  a sequence over  $G$ with $|S|=n+r$, where $r \in [0, \mathsf D(G)-2]$. An inverse problem associated with the Gao Theorem asks that if $0\not\in  \Sigma_{n}(S)$, what is the lower bound of $|\Sigma_{n}(S)|$? In 1999, Bollob\'{a}s and Leader \cite{BL99} proved that $|\Sigma_{n}(S)|\geq r+1$. Yu \cite[Theorem 2]{Yu2004} and Xia et al. \cite[Theorem 1.1]{XQQ2014}  proved independently that $|\Sigma_{n}(S)|= r+1$ if and only if $|\mathrm{supp}(S)|=2$. In 2016, Gao, Grynkiewicz, and Xia \cite{GGX2016} proved the following result, which confirmed a conjecture of Hamidoune \cite{Ham2003}.

\begin{theorem}\cite[Theorem 1.4]{GGX2016}\label{GGXNZSFS}
Let $G$ be an abelian group of order $n$ and $S$ a sequence over $G$ with $|S|=n+r$, where $r \in [1, \mathsf D(G)-2]$. Suppose $0\not\in \Sigma_{n}(S)$. Then $|\Sigma_{n}(S)|\geq r+|\mathrm{supp}(S)|-1$.
\end{theorem}

Moreover, Gao, Grynkiewicz, and Xia \cite{GGX2016} proposed the following conjecture.

\begin{conjecture}\cite[Conjecture 4.5]{GGX2016}\label{conjecture of GGX}
Let $G$ be an abelian group of order $n$ and $S$ a sequence over $G$ with $|S|=n+r$, where $r \in [1, \mathsf D(G)-2]$. Suppose $0\not\in \Sigma_{n}(S)$. Then there exists a zero-sum free sequence $T$ over $G$ with $|T|=r+1$ such that $|\Sigma_{n}(S)| \geq |\Sigma(T)|$ and $|\mathrm{supp}(T)|\geq \min\{r+1, |\mathrm{supp}(S)|-1\}$.
\end{conjecture}

In this paper, we shall determine when the equality holds in Theorem~\ref{GHZSFS}, and give a counterexample of Conjecture~\ref{conjecture of GGX}. Moreover, we give a generalization of Theorem~\ref{GHZSFS}, by which we obtain an alternative proof of Theorem~\ref{GGXNZSFS}. Our main results state as following (See Section 2 for the details of the notation).

\begin{theorem}\label{mainresults1}
Let $G$ be a finite abelian group and $S$ a zero-sum free sequence over $G$. Then $|\Sigma(S)|= |S|+|\mathrm{supp}(S)|-1$ if and only if $S$ is one of the following forms.
\begin{itemize}
\item[(1)] $S=g^{|S|}$ for some $g\in G\setminus\{0\}$ with $\ord(g)\geq |S|+1$.
\item[(2)] $S=g_1\cdot g_2$ for some $g_1, g_2\in G\setminus\{0\}$ with $g_1\neq g_2$.
\item[(3)] $S=g^{|S|-1}\cdot(2g)$ for some $g\in G\setminus\{0\}$ with $|S|\geq 3$ and $\ord(g)\geq |S|+2$.
\item[(4)] $S=g_1\cdot g_2 \cdot (g_1+g_2)$, where $g_1, g_2\in G\setminus\{0\}$, $2g_1\neq 0$, $2g_2= 0$ and $2g_1+g_2\neq 0$.
\end{itemize}
\end{theorem}

\begin{theorem}\label{mainresults2}
Let $n\geq 7$ be an odd number. Let $G=\langle g\rangle$ be a cyclic group of order $n$. Let $S=0^{n-3}\cdot g^{n-3}\cdot (2g)\cdot ((n-1)g)$ and $r=|S|-n=n-4$. Then $|S|=n+r$ with $0\not\in \Sigma_{n}(S)$, and for every zero-sum free sequence $T$ over $G$ with $|T|=r+1$, we have either $|\Sigma_{n}(S)| < |\Sigma(T)|$ or $|\mathrm{supp}(T)|< \min\{r+1, |\mathrm{supp}(S)|-1\}$.
\end{theorem}

By Theorem~\ref{mainresults2}, we obtain that Conjecture~\ref{conjecture of GGX} does not always hold.

\begin{theorem}\label{mainresults3}
Let $G$ be a finite abelian group and $S$ a sequence over $G$ of length $|S| \geq k+1$, where $k \geq 1$ is a positive integer. Suppose $0 \not\in \Sigma_{\geq k}(S)$ and $0\not\in \mathrm{supp}(S)$, then  $|\Sigma_{\geq k}(S)|\geq |S|-k+|\mathrm{supp}(S)|$.
\end{theorem}

By using Theorem~\ref{mainresults3}, we can give an alternative proof of Theorem~\ref{GGXNZSFS}. If $|S|-k=1$, we have $|\Sigma_{\geq k}(S)|= |S|-k+|\mathrm{supp}(S)|$ (see the proof of Theorem~\ref{mainresults3}), so the lower bound is sharp in  Theorem~\ref{mainresults3}.

The rest of the paper is organized in the following way. Section 2 provides the basic notation and some preliminary results. In Section 3, we prove Theorems~\ref{mainresults1} and \ref{mainresults2}. In Section 4, we prove Theorems~\ref{mainresults3} and \ref{GGXNZSFS}. In the last section, we provide some concluding remarks.

\section{Notation and Preliminary results}

\subsection{Notation}

Let $C_n$ denote the cyclic group of $n$ elements. Every non-trivial finite abelian group $G$ can be written in the form $G=C_{n_1}\oplus \ldots \oplus C_{n_r}$ with $1 < n_1 \mid \ldots \mid n_r$. We consider sequences over $G$ as elements in the free abelian monoid with basis $G$. So a sequence $S$ over $G$ can be written in the form
\begin{equation*}
S=g_1 \cdot  \ldots \cdot g_\ell=\prod_{g \in G} g^{\mathsf v_g (S)},
\end{equation*}
where $\mathsf v_g(S) \in \mathbb{N}\cup\{0\}$ denotes the {\it multiplicity} of $g$ in $S$. We call $|S|= \ell = \sum_{g\in G}\mathsf v_g(S) \in \mathbb{N}\cup\{0\}$ the {\it length} of $S$,  $\sigma(S)=\sum_{i=1}^{\ell}g_i = \sum_{g\in G}\mathsf v_g(S)g\in G $  the {\it sum } of $S$, and  $\mathrm{supp}(S)=\{g \in G \mid \mathsf v_g(S) >0\}$ the {\it support} of $S$. Let $\mathsf h(S)=\max_{g\in G} \mathsf v_g(S)$ denote the {\it maximal occurrence} of elements in $S$.

A sequence $T$ is called a {\it subsequence} of $S$ if $\mathsf v_g (T) \le \mathsf v_g (S)$ for all $g \in G$. When $T$ is a subsequence of $S$, let $ST^{-1}$ denote the subsequence of $S$ with $T$ deleted from $S$. If $S_1$ and $S_2$ are two sequences over $G$,  let $S_1S_2$ denote the sequence satisfying that $\mathsf v_g(S_1S_2)= \mathsf v_g(S_1) + \mathsf v_g(S_2)$ for all $g \in G$.

Let $S=g_1 \cdot  \ldots \cdot g_\ell$ be a sequence over a finite abelian group $G$. For an element $g\in G$, let $g+S$ denote the sequence $(g+g_1) \cdot  \ldots \cdot (g+g_\ell)$ over $G$. Suppose $A$ and $B$ are two subsets of $G$. Let $A+B=\{a+b \mid a\in A, b\in B\}$. If $A=\{a\}$, we denote $A+B$ by $a+B$ for short. Let $A\setminus B=\{a \mid a\in A, a\not\in B\}$.

\subsection{Preliminary results}
We need a series of preliminary results. The first result is due to Bovey, Erd\H{o}s, and Niven \cite{BEN75}, and it also can be found in \cite[Theorem 5.3.1]{GH06}.

\begin{lemma}\label{lemma of BEN}
Let $G$ be a finite abelian group and $S$ a zero-sum free sequence over $G$. Suppose $S_1, S_2, \ldots, S_t$ are disjoint subsequences of $S$. Then
\begin{equation*}
|\Sigma(S)|\geq |\Sigma(S_1)|+|\Sigma(S_2)|+\cdots+|\Sigma(S_t)|.
\end{equation*}
\end{lemma}

\begin{lemma}\cite[Proposition 5.3.2]{GH06}\label{lemma of GH}
Let $G$ be a finite abelian group and $S$ a zero-sum free subset of $G$.
\begin{itemize}
\item[1.] If $|S|=1$ or $2$, then $|\Sigma(S)|=2|S|-1.$
\item[2.] If $|S|=3$, then
        \begin{itemize}
        \item[] either $|\Sigma(S)|\geq 6=2|S|$
        \item[] or $S=g_1\cdot g_2 \cdot (g_1+g_2)$, where $g_1, g_2\in G\setminus\{0\}$, $2g_2=0$, $2g_1\neq 0$ and $g_2+2g_1\neq 0$. In this case, $\Sigma(S)=\{g_1, g_2, g_1+g_2, 2g_1, 2g_1+g_2\}$ and $|\Sigma(S)|=5=2|S|-1$.
        \end{itemize}
\item[3.] If $|S|\geq 4$, then $|\Sigma(S)|\geq 2|S|.$
\end{itemize}
\end{lemma}

\begin{lemma}\label{lemma of ZSFS}
Let $G$ be a finite abelian group and $S$ a zero-sum free sequence over $G$. Then $|\Sigma(S)|\geq |S|$, and equality holds if and only if $S=g^{|S|}$ for some $g\in G$ with $\ord(G)\geq |S|+1$.
\end{lemma}

\begin{proof}
The result follows from Lemma~\ref{lemma of BEN} and Lemma~\ref{lemma of GH} (1).
\end{proof}

The following result was first proved by Pixton \cite{Pixton}.

\begin{lemma}\label{pixton 1}\cite[Lemma 2.9]{Pixton}
Let $G$ be a finite abelian group and let $X\subset G\setminus\{0\}$ be a generating  set for $G$. Suppose $Y$ is a nonempty proper subset of $G$. Then
\begin{equation*}
\sum_{g\in X} |(Y+g)\setminus Y|\geq |X|.
\end{equation*}
\end{lemma}

The following result is due to Yuan \cite{Y092}.

\begin{lemma}\cite[Lemma 2.9]{Y092}\label{m^2}
Let $m$ be a positive integer and $G$  a finite abelian group, and let $X\subset G\setminus \{0\}$ be a generating set for $G$. Suppose that $A \subset G$ satisfies $|(A+g)\setminus A|\leq m$ for all $g\in X$ and there exists a proper subset $Y \subset X$ such that $|\langle Y\rangle|>m$ and $|G/\langle Y\rangle|>m$. Then $\min\{|A|, |G\setminus A|\}\leq m^2.$
\end{lemma}

We also need the following result due to Olson.

\begin{lemma}\cite[Lemma 1]{MO1967}\label{cyclic}
Let $m$ be a positive integer. Let $A=\{a_0+\lambda a \mid \lambda =0,1,\ldots,s\}$ be a set of $G=C_m$, where $a_0, a \in G$, $\ord(a)=m$, and $1\leq s\leq m-3$. If $A=\{b_0+\lambda b \mid \lambda =0,1,\ldots,s\}$ for some $b_0, b\in G$, then $a=\pm b$.
\end{lemma}


\begin{lemma}\label{sumofelementsoforder2}
Let $G=C_2^n$, where $n\geq 2$. Then $G \setminus \{0\}$ is a zero-sum set.
\end{lemma}
\begin{proof}
Let $e_1, e_2, \ldots, e_n$ be a basis of $C_2^n$. Then $\sigma(G\setminus \{0\})=\sigma(G)=2^{n-1}(e_1+e_2+\cdots+e_n)=0,$  and we are done.
\end{proof}

The following result was proved by Gao, Grynkiewicz, and Xia \cite{GGX2016}.
\begin{lemma}\cite[Lemma 2.5]{GGX2016}\label{Lemma 2.5}
Let $G$ be an abelian group of order $n$ and $k\leq n$ a nonnegative integer. Let $S=0^{n-k}T$ be a sequence over $G$ with $|T|\geq k $ and $\mathsf v_{0}(S)\geq \mathsf h(S)-1$. Then
\begin{center}
$\Sigma_{\geq k}(T)=\Sigma_n(S)$.
\end{center}
\end{lemma}

\section{Proofs of Theorems~\ref{mainresults1} and \ref{mainresults2}}

\noindent{\bf Proof of Theorem~\ref{mainresults1}}

If $S$ is one of the forms in (1)-(4), it is easy to verify that $|\Sigma(S)|=|S|+|\mathrm{supp}(S)|-1$.

Next we assume that $S$ is a zero-sum free sequence over $G$ with $|\Sigma(S)|=|S|+|\mathrm{supp}(S)|-1$.

Let $S_1$ be a subsequence of $S$ such that $\mathrm{supp}(S_1)=\mathrm{supp}(S)$ and $|S_1|=|\mathrm{supp}(S)|$. Let $S_2=SS_1^{-1}$. Then
\begin{center}
$S=S_1S_2$, $|S|=|S_1|+|S_2|$, and $\mathrm{supp}(S_2)\subset \mathrm{supp}(S_1)$.
\end{center}

By Lemma~\ref{lemma of BEN}, we obtain that
\begin{align*}
|\Sigma(S)|&=|\Sigma(S_1S_2)|\geq |\Sigma(S_1)|+|\Sigma(S_2)|.
\end{align*}

We first prove that $|\mathrm{supp}(S_2)|\leq 1$, and thus $|\Sigma(S_2)|=|S_2|$. By Theorem~\ref{GHZSFS}, we obtain that
\begin{align*}
|\Sigma(S)|&\geq |\Sigma(S_1)|+|\Sigma(S_2)| \\
&\geq (|S_1|+|\mathrm{supp}(S_1)|-1)+(|S_2|+|\mathrm{supp}(S_2)|-1) \\
&=(|S_1|+|S_2|+|\mathrm{supp}(S_1)|-1)+(|\mathrm{supp}(S_2)|-1)\\
&=(|S|+|\mathrm{supp}(S)|-1)+(|\mathrm{supp}(S_2)|-1).
\end{align*}
Since $|\Sigma(S)|=|S|+|\mathrm{supp}(S)|-1$, we infer that $|\mathrm{supp}(S_2)|\leq 1$. In view of Lemma~\ref{lemma of ZSFS}, we obtain that $|\Sigma(S_2)|=|S_2|$.

We next prove that $|\Sigma(S_1)|\leq 2|S_1|-1$. Assume to the contrary that $|\Sigma(S_1)|\geq 2|S_1|$. Then we obtain that
\begin{align*}
|\Sigma(S)|\geq |\Sigma(S_1)|+|\Sigma(S_2)| \geq 2|S_1|+|S_2|=|S|+|\mathrm{supp}(S)|,
\end{align*}
yielding a contradiction to that $|\Sigma(S)|=|S|+|\mathrm{supp}(S)|-1$. So $$|\Sigma(S_1)|\leq 2|S_1|-1.$$
It follows from Lemma~\ref{lemma of GH} that either $|S_1|=|\mathrm{supp}(S)|\leq 2$ or $S_1=g_1\cdot g_2 \cdot (g_1+g_2)$, where $g_1, g_2\in G\setminus\{0\}$, $2g_2=0$, $2g_1\neq 0$ and $2g_1+g_2\neq 0$.

We distinguish there cases.

{\bf Case 1.} $|S_1|=|\mathrm{supp}(S)|= 1$. Then $S=g^{|S|}$ for some $g\in G$. In this case $\Sigma(S)=\{g,2g,\ldots,|S|g\}$. Hence $|\Sigma(S)|=|S|=|S|+|\mathrm{supp}(S)|- 1$. Since $S$ is a zero-sum free sequence, we obtain that $\ord(g)\geq |S|+1$. So $S$ is of form (1).

{\bf Case 2.} $|S_1|=|\mathrm{supp}(S)|= 2$. Then $S_1=g_1\cdot g_2$ for some $g_1, g_2 \in G$ with $g_1\neq g_2$.

If $|\mathrm{supp}(S_2)|=0$, then $S=S_1=g_1\cdot g_2$, and thus $S$ is of form (2).

Next we assume that $|\mathrm{supp}(S_2)|=1$. Without loss of generality, we assume that $S_2=g_1^{|S|-2}$ with $|S|\geq 3$. So $S=g_1^{|S|-1}\cdot g_2$, and we will prove that $g_2=2g_1$. Since $S$ is a zero-sum free sequence, we infer that $$g_1, 2g_1, \ldots, (|S|-1)g_1, (|S|-2)g_1+g_2, (|S|-1)g_1+g_2\in \Sigma(S)$$ are pairwise distinct. Since $|\Sigma(S)|=|S|+|\mathrm{supp}(S)|-1=|S|+1$, we obtain that
$$\Sigma(S)=\{g_1, 2g_1, \ldots, (|S|-1)g_1, (|S|-2)g_1+g_2, (|S|-1)g_1+g_2\}.$$ Note that
$$(|S|-3)g_1+g_2\in \Sigma(S).$$ Since $S$ is a zero-sum free sequence, we infer that $(|S|-3)g_1+g_2=(|S|-1)g_1$, and thus $g_2=2g_1$. So $S$ is of form (3) by taking $g=g_1$.

{\bf Case 3.} $S_1=g_1\cdot g_2 \cdot (g_1+g_2)$, where $g_1, g_2\in G\setminus\{0\}$, $2g_2=0$, $2g_1\neq 0$ and $2g_1+g_2\neq 0$.

We will prove that $|\mathrm{supp}(S_2)|=0$, then $S=S_1=g_1\cdot g_2 \cdot (g_1+g_2)$, and thus $S$ is of form (4).

Assume to the contrary that $|\mathrm{supp}(S_2)|=1$. Since $S$ is zero-sum free, we infer that $g_2\not\in \mathrm{supp}(S_2)$. So $\mathrm{supp}(S_2) \subset \{g_1, g_1+g_2\}$. Without loss of generality, we assume that $S_2=g_1^{|S|-3}$ with $|S|\geq 4$.  So $S=g_1^{|S|-2}\cdot g_2\cdot (g_1+g_2)$ and $|\Sigma(S)|=|S|+|\mathrm{supp}(S)|-1=|S|+2$. Since $S$ is a zero-sum free sequence and $2g_2=0$, we infer that
\begin{align*}
g_1, 2g_1, \ldots, (|S|-1)g_1,&  (|S|-4)g_1+g_2, (|S|-3)g_1+g_2, \\ (|S|-2)g_1+g_2,& (|S|-1)g_1+g_2\in \Sigma(S)
\end{align*}
are pairwise distinct. It follows that $|\Sigma(S)|\geq |S|+3> |S|+|\mathrm{supp}(S)|-1$, yielding a contradiction. Therefore, $|\mathrm{supp}(S_2)|=0$, and we are done.

This completes the proof. \qed

\bigskip

\noindent{\bf Proof of Theorem~\ref{mainresults2}}

Since $S=0^{n-3}\cdot g^{n-3}\cdot (2g)\cdot ((n-1)g)$, we infer that  $|S|=2n-4=n+r$ with $r=n-4$, and
$$\Sigma_{n}(S)=\{g, 2g, \ldots, (n-1)g\}.$$ So $0\not\in \Sigma_{n}(S)$ and $|\Sigma_{n}(S)|=n-1.$

Assume to the contrary that there exists a zero-sum free sequence $T$ over $G$ with $|T|=r+1=n-3$ such that $$|\Sigma_{n}(S)| \geq |\Sigma(T)|$$ and $$|\mathrm{supp}(T)|\geq \min\{r+1, |\mathrm{supp}(S)|-1\}.$$
Since $ |\mathrm{supp}(S)|=4$, we have that
\begin{align*}
|\mathrm{supp}(T)|\geq \min\{r+1, |\mathrm{supp}(S)|-1\}=\min\{n-3, 3\}=3.
\end{align*}
By Theorem~\ref{GHZSFS}, we obtain that
\begin{align*}
|\Sigma(T)|\geq |T|+|\mathrm{supp}(T)|-1\geq n-3+3-1=n-1.
\end{align*}
Since $|\Sigma(T)|\leq |\Sigma_{n}(S)|=n-1$, we obtain that
\begin{equation*}
|\Sigma(T)|=|T|+|\mathrm{supp}(T)|-1=n-1,
\end{equation*}
and thus $|\mathrm{supp}(T)|=3$. By Theorem~\ref{mainresults1} we infer that $T$ is of form $T=g_1\cdot g_2 \cdot (g_1+g_2)$, where $g_1, g_2\in G\setminus\{0\}$, $2g_2=0$, $2g_1\neq 0$ and $g_2+2g_1\neq 0$. But $n$ is odd, yielding a contradiction to that $2g_2=0$.

This completes the proof. \qed

\section{Proofs of Theorems~\ref{mainresults3} and \ref{GGXNZSFS}}

\noindent{\bf Proof of Theorem~\ref{mainresults3}}

\begin{proof}

Without loss of generality we assume that $G=\langle \mathrm{supp}(S)\rangle$. Let $r=|S|-k$. We will prove the theorem by induction on $r$.

\medskip
Suppose $r=1$. Then $|S|=k+1$, and we have that
\begin{equation*}
\Sigma_{\geq k}(S)=\Sigma_{k}(S)\cup \Sigma_{ k+1}(S)=(\sigma(S)-\mathrm{supp}(S))\cup \{\sigma(S)\}.
\end{equation*}
Since $0\not\in \mathrm{supp}(S)$, we infer that
\begin{equation*}
(\sigma(S)-\mathrm{supp}(S))\cap \{\sigma(S)\}=\emptyset.
\end{equation*}
Hence
\begin{align*}
|\Sigma_{\geq k}(S)| =|\sigma(S)-\mathrm{supp}(S)| + |\{\sigma(S)\}| = 1 + |\mathrm{supp}(S)|.
\end{align*}
We are done.

\medskip
Assume that the theorem holds for integer $r$, where $r\geq 1$.

\medskip
Next we suppose that $|S|=k+r+1$, with $0\not\in \mathrm{supp}(S)$ and $0\not\in \Sigma_{\geq k}(S)$. We will prove that $|\Sigma_{\geq k}(S)|\geq r+1+|\mathrm{supp}(S)|$. Assume to the contrary that
\begin{equation}\label{upbound}
|\Sigma_{\geq k}(S)|\leq  r+|\mathrm{supp}(S)|.
\end{equation}

For every $g\in \mathrm{supp}(S)$, we have that
\begin{equation}\label{subset}
\left(g+\Sigma_{\geq k}(Sg^{-1})\right)\cup \left(\Sigma_{\geq k}(Sg^{-1})\right)\subset \Sigma_{\geq k}(S).
\end{equation}
So
\begin{equation*}
|\Sigma_{\geq k}(S)| \geq |\Sigma_{\geq k}(Sg^{-1})|.
\end{equation*}

\medskip

{\bf Claim 1.} If $|\Sigma_{\geq k}(S)|=|\Sigma_{\geq k}(Sg^{-1})|$ for some $g\in \mathrm{supp}(S)$, then $|(\Sigma_{\geq k}(S)+g)\setminus \Sigma_{\geq k}(S)|=0$.

{\it Proof of Claim 1.}
If $|\Sigma_{\geq k}(S)|=|\Sigma_{\geq k}(Sg^{-1})|$ for some $g\in \mathrm{supp}(S)$, we infer that
\begin{equation*}
|\Sigma_{\geq k}(S)| = |\Sigma_{\geq k}(Sg^{-1})| = |g+\Sigma_{\geq k}(Sg^{-1})|.
\end{equation*}
By \eqref{subset}, we have
\begin{equation*}
\Sigma_{\geq k}(S)= g+\Sigma_{\geq k}(Sg^{-1})=\Sigma_{\geq k}(Sg^{-1}).
\end{equation*}
It follows that $\Sigma_{\geq k}(Sg^{-1})$ is a union of cosets of $\langle g\rangle$, and so is $\Sigma_{\geq k}(S)$. So
$\Sigma_{\geq k}(S)+g= \Sigma_{\geq k}(S)$, and this proves Claim 1.

\medskip

{\bf Claim 2.} If $|\Sigma_{\geq k}(S)| > |\Sigma_{\geq k}(Sg^{-1})|$ for some $g\in \mathrm{supp}(S)$, then
\begin{itemize}
\item[(1)] $\mathsf v_g(S)=1$.
\item[(2)] $|\Sigma_{\geq k}(S)|=r+|\mathrm{supp}(S)|$;
\item[(3)] $|\Sigma_{\geq k}(Sg^{-1})|=r+|\mathrm{supp}(S)|-1$;
\item[(4)] $|(\Sigma_{\geq k}(S)+g) \setminus \Sigma_{\geq k}(S) | \leq 1$.
\end{itemize}

{\it Proof of Claim 2.}  We first prove that $\mathsf v_g(S)=1$. Assume to the contrary that $\mathsf v_g(S)\geq 2$, we infer that $0\not\in \mathrm{supp}(Sg^{-1})=\mathrm{supp}(S)$ and  $0\not\in \Sigma_{\geq k}(Sg^{-1})$. By the induction hypothesis, we obtain that
\begin{equation}\label{lowerbound1}
|\Sigma_{\geq k}(Sg^{-1})|\geq r+|\mathrm{supp}(Sg^{-1})|=r+|\mathrm{supp}(S)|.
\end{equation}
by \eqref{upbound}, \eqref{subset} and \eqref{lowerbound1}, we infer that
\begin{equation*}
|\Sigma_{\geq k}(S)|= |g+\Sigma_{\geq k}(Sg^{-1})|=|\Sigma_{\geq k}(Sg^{-1})|=r+|\mathrm{supp}(S)|,
\end{equation*}
yielding a contradiction to that $|\Sigma_{\geq k}(S)| > |\Sigma_{\geq k}(Sg^{-1})|$. Therefore, $\mathsf v_g(S)=1$. This proves (1).

Since $\mathsf v_g(S)= 1$, we infer that $0\not\in \mathrm{supp}(Sg^{-1})$, $|\mathrm{supp}(Sg^{-1})|=|\mathrm{supp}(S)|-1$, and  $0\not\in \Sigma_{\geq k}(Sg^{-1})$. By the induction hypothesis, we obtain that
\begin{equation}\label{lowerbound2}
|\Sigma_{\geq k}(Sg^{-1})|\geq r+|\mathrm{supp}(Sg^{-1})|=r+|\mathrm{supp}(S)|-1.
\end{equation}
Since $|\Sigma_{\geq k}(S)| > |\Sigma_{\geq k}(Sg^{-1})|$, by \eqref{upbound}, \eqref{subset} and \eqref{lowerbound2}, we infer that
\begin{equation}\label{equation1}
 |\Sigma_{\geq k}(S)|=r+|\mathrm{supp}(S)| \mbox{ and } |\Sigma_{\geq k}(Sg^{-1})|=r+|\mathrm{supp}(S)|-1.
\end{equation}
This proves (2) and (3).

It follows from \eqref{equation1} that
\begin{equation}\label{equation2}
|\Sigma_{\geq k}(S)|= |g+\Sigma_{\geq k}(Sg^{-1})|+1=|\Sigma_{\geq k}(Sg^{-1})|+1.
\end{equation}

If $g+\Sigma_{\geq k}(Sg^{-1})\neq \Sigma_{\geq k}(Sg^{-1})$, by \eqref{subset} and \eqref{equation2}, we infer that
\begin{equation*}
\Sigma_{\geq k}(S)=\left(g+\Sigma_{\geq k}(Sg^{-1})\right)\cup \left(\Sigma_{\geq k}(Sg^{-1})\right),
\end{equation*}
and
\begin{equation*}
\left|\left(g+\Sigma_{\geq k}(Sg^{-1})\right)\setminus \Sigma_{\geq k}(Sg^{-1})\right|=1.
\end{equation*}
These imply that $\Sigma_{\geq k}(Sg^{-1})$ is a union of  some cosets of $\langle g\rangle$ and an arithmetic progression with difference $g$, and so is $\Sigma_{\geq k}(S)$. Therefore,
\begin{equation*}
|\left(\Sigma_{\geq k}(S)+g\right)\setminus \Sigma_{\geq k}(S)|\leq 1.
\end{equation*}

Next we assume that $g+\Sigma_{\geq k}(Sg^{-1})= \Sigma_{\geq k}(Sg^{-1})$. By \eqref{subset} and \eqref{equation2}, we have that
\begin{equation}\label{equation3}
\Sigma_{\geq k}(S)=\Sigma_{\geq k}(Sg^{-1})\cup \{h\}
\end{equation}
for some $h\in G\setminus \left(\Sigma_{\geq k}(Sg^{-1})\right)$. Since $g+\Sigma_{\geq k}(Sg^{-1})= \Sigma_{\geq k}(Sg^{-1})$, we infer that $\Sigma_{\geq k}(Sg^{-1})$ is a union of cosets of $\langle g\rangle$, so
\begin{equation}\label{equation4}
\Sigma_{\geq k}(S)+g=\left(\Sigma_{\geq k}(Sg^{-1})+g\right)\cup \{h+g\}=\Sigma_{\geq k}(Sg^{-1})\cup \{h+g\}.
\end{equation}
Combine \eqref{equation3} and \eqref{equation4}, we obtain that
\begin{equation*}
|\left(\Sigma_{\geq k}(S)+g\right)\setminus \Sigma_{\geq k}(S)|=1.
\end{equation*}
This proves (4).

This completes the proof of Claim 2.

\medskip

{\bf Claim 3.} $\mathsf v_g(S)=1$ for every $g\in \mathrm{supp}(S)$, $|\Sigma_{\geq k}(S)|=r+|\mathrm{supp}(S)|$, and $|\left(\Sigma_{\geq k}(S)+g\right)\setminus \Sigma_{\geq k}(S)|=1$ for every $g\in \mathrm{supp}(S)$.

{\it Proof of Claim 3.} Applying Lemma~\ref{pixton 1} with $X=\mathrm{supp}(S)$ and $Y=\Sigma_{\geq k}(S)$, we infer that
\begin{equation*}
\sum_{g\in \mathrm{supp}(S)}|\left(\Sigma_{\geq k}(S)+g\right)\setminus \Sigma_{\geq k}(S)|\geq |\mathrm{supp}(S)|.
\end{equation*}
On the other hand, by Claim 1 and Claim 2, we obtain that
\begin{equation*}
\sum_{g\in \mathrm{supp}(S)}|\left(\Sigma_{\geq k}(S)+g\right)\setminus \Sigma_{\geq k}(S)|\leq \sum_{g\in \mathrm{supp}(S)} 1= |\mathrm{supp}(S)|.
\end{equation*}
So
\begin{equation*}
\sum_{g\in \mathrm{supp}(S)}|\left(\Sigma_{\geq k}(S)+g\right)\setminus \Sigma_{\geq k}(S)|= |\mathrm{supp}(S)|,
\end{equation*}
and thus
\begin{equation*}
|\left(\Sigma_{\geq k}(S)+g\right)\setminus \Sigma_{\geq k}(S)|=1,
\end{equation*}
for every $g\in \mathrm{supp}(S)$.

By Claim 1 and Claim 2, we obtain that $\mathsf v_g(S)=1$ for every $g\in \mathrm{supp}(S)$ and $|\Sigma_{\geq k}(S)|=r+|\mathrm{supp}(S)|$. This proves Claim 3.

\medskip
{\bf Claim 4.} $|G\setminus \Sigma_{\geq k}(S)|=1$.

{\it Proof of Claim 4.} Since $0\not\in \Sigma_{\geq k}(S)$, we infer that
\begin{equation*}
|G\setminus \Sigma_{\geq k}(S)|\geq 1.
\end{equation*}

By Claim 3, we have that $|\left(\Sigma_{\geq k}(S)+g\right)\setminus \Sigma_{\geq k}(S)|=1$ for every $g\in \mathrm{supp}(S)$.

Suppose there exists a subset $Y\subset S$ such that $|\langle Y\rangle|> 1$ and $|G/\langle Y\rangle| > 1$. By Lemma~\ref{m^2}, we obtain that
\begin{equation*}
\min\{|\Sigma_{\geq k}(S)|, |G\setminus \Sigma_{\geq k}(S)|\}\leq 1.
\end{equation*}
By Claim 3, we have
\begin{equation*}
|\Sigma_{\geq k}(S)| = r+|\mathrm{supp}(S)| = r+(k+r+1) = k+2r+1 \geq 4.
\end{equation*}
So $|G\setminus \Sigma_{\geq k}(S)|\leq 1$, and thus $|G\setminus \Sigma_{\geq k}(S)|=1$. We are done.

Next we assume that there does not exist a subset $Y\subset S$ such that $|\langle Y\rangle|> 1$ and $|G/\langle Y\rangle|> 1$. It follows that $G$ is a cyclic group and every element of $\mathrm{supp}(S)$ is of order $|G|$. Since $|\left(\Sigma_{\geq k}(S)+g\right)\setminus \Sigma_{\geq k}(S)|=1$ for every $g\in \mathrm{supp}(S)$ and $|\Sigma_{\geq k}(S)|  \geq 4$, we infer that $\Sigma(S)$ is an arithmetic progression with difference $g$ for every $g\in \mathrm{supp}(S)$. If $|\Sigma_{\geq k}(S)|\leq |G|-2$, then by Lemma~\ref{cyclic} we obtain that $|\mathrm{supp}(S)|\leq 2$, yielding a contradiction to that $|\mathrm{supp}(S)|=|S|=k+r+1\geq 3$. Therefore, $|\Sigma_{\geq k}(S)|= |G|-1$. So $|G\setminus \Sigma_{\geq k}(S)|=1$. We are done.

This proves Claim 4.

\medskip

By Claim 3 and Claim 4, we have that
\begin{equation*}\label{sizeofG}
|G|=|\Sigma_{\geq k}(S)|+1=r+|\mathrm{supp}(S)|+1=k+2r+2.
\end{equation*}

Write $S=S_2S_3$ such that $S_2$ contains all the elements of order $2$ in $S$ and $S_3$ contains all the elements of order at least $3$ in $S$. Since $\mathsf v_g(S)=1$ for every $g\in \mathrm{supp}(S)$, we infer that $S, S_2, S_3$ are subsets of $G$.

Let $S_i'$ be the largest zero-sum subsequence (in length) of $S_i$ ($|S_i'|=0$ if $S_i$ is zero-sum free), where $i=2,3$. Then $S_2'S_3'$ is a zero-sum subsequence of $S$. Since $0\not\in \Sigma_{\geq k}(S)$, we infer that
\begin{equation}\label{sizeofzero-sumsubsequence}
|S_2'|+|S_3'| = |S_2'S_3'| < k.
\end{equation}

Let $x=|S_2|$ and $y=|S_2|-|S_2'|=|S_2S_2'^{-1}|$. Then
\begin{equation*}
|S_3|=|S|-|S_2|=k+r+1-x.
\end{equation*}
Let $t$ be the largest integer such that $C_2^t$ is a subgroup of $G$ ($t=0$ if $|G|$ is odd). So $S_2$ is subset of $C_2^t$ with $0\not\in \mathrm{supp}(S_2)$. So $x\leq 2^t-1$.

\medskip
{\bf Claim 5.} $x+y>2^t$ and $t\geq 2$.

{\it Proof of Claim 5.} If $|S_3| \leq \frac{|G|-|C_2^t|}{2}$, i.e. $k+r+1-x<\frac{k+2r+2-2^t}{2}$, then we have
\begin{equation}\label{2x-k}
2x-k\geq 2^t.
\end{equation}
By \eqref{sizeofzero-sumsubsequence}, we infer that $x-y=|S_2'|< k.$ This together with \eqref{2x-k} yields that $x+y > 2^t$.  By \eqref{2x-k}, we obtain $x\geq 2^{t-1}+k/2 > 2^{t-1}$. Since $x\leq 2^t-1$, we infer that $t\geq 2$, and we are done.

Next we assume that
\begin{equation*}
|S_3| > \frac{|G|-|C_2^t|}{2}.
\end{equation*}
Note that we can divide the set $G\setminus C_2^t$ into $\frac{|G|-|C_2^t|}{2}$ subsets of form $\{g, -g\}$. Since $S_3$ is a subset of $G\setminus C_2^t$, we obtain that $S_3$ contains at least
\begin{align*}
|S_3|-\frac{|G|-|C_2^t|}{2}=(k+r+1-x)-\frac{k+2r+2-2^t}{2}=\frac{k-2x+2^t}{2}
\end{align*}
subsets of form $\{g,-g\}$. Since each subset of form $\{g, -g\}$ is a zero-sum set, we infer that
\begin{equation}\label{k}
|S_3'| \geq 2 \cdot \frac{k-2x+2^t}{2}=k-2x+2^t.
\end{equation}
Since $|S_2'|=x-y$, by \eqref{sizeofzero-sumsubsequence} and \eqref{k}, we infer that
\begin{equation*}
k> |S_2'|+|S_3'|\geq (x-y)+(k-2x+2^t).
\end{equation*}
It follows that
\begin{equation*}\label{x+y}
x+y> 2^t.
\end{equation*}
Also by \eqref{sizeofzero-sumsubsequence} and \eqref{k}, we obtain that $k-2x+2^t\leq |S_3|<k$. It follows that $x > 2^{t-1}$. Since $x\leq 2^t-1$,  we infer that $t\geq 2$, and we are done. This proves Claim 5.

\medskip
If $x=2^t-1$, then $S_2=C_2^t\setminus \{0\}$. Since $t\geq 2$, by Lemma~\ref{sumofelementsoforder2}, we obtain that $S_2$ is a zero-sum subset of $C_2^t$. By the definition of $y$, we obtain that $y=0$, and thus $x+y=2^t-1<2^t$, yielding a contradiction to Claim 5.

If $x=2^t-2$, then $S_2=C_2^t \setminus \{0, h\}$ for some $h\neq 0$. Since $t\geq 2$, we can find $h_1, h_2\in S_2$ such that $h_1+h_2=h$. By Lemma~\ref{sumofelementsoforder2}, we obtain that $S_2h$ is a zero-sum subset of $C_2^t$. So $S_2 h_1^{-1}h_2^{-1}$ is a zero-sum subset of $C_2^t$. By the definition of $y$, we obtain that $y\leq 2$. So $x+y\leq 2^t-2+2=2^t$,  yielding a contradiction to Claim 5.

Next we assume that $x\leq 2^t-3$. By Claim 5, we obtain that $y > 2^t-x\geq 3$. So $|S_2S_2'^{-1}|=y\geq 4$. Since $S_2S_2'^{-1}$ is a zero-sum free set, by Lemma~\ref{lemma of GH}.3, we obtain that
\begin{equation*}
|\Sigma(S_2S_2'^{-1})|\geq 2|S_2S_2'^{-1}|=2y.
\end{equation*}
Since $\Sigma(S_2S_2'^{-1}), S_2' \subset C_2^t$ and
\begin{equation*}
|\Sigma(S_2S_2'^{-1})|+|S_2'| \geq 2y+(x-y)=x+y > 2^t,
\end{equation*}
we infer that
\begin{equation*}
\Sigma(S_2S_2'^{-1}) \cap S_2' \neq \emptyset.
\end{equation*}
Moreover, since $S_2$ is a set, we infer that $S_2S_2'^{-1} \cap S_2'=\emptyset.$ So
\begin{equation*}
\Sigma_{\geq 2}(S_2S_2'^{-1}) \cap S_2' \neq \emptyset.
\end{equation*}
This implies that there exist $h\in S_2'$ and a subsequence $S_2''$ of $S_2S_2'^{-1}$ such that $h=\sigma(S_2'')$ and $|S_2''|\geq 2$. Then $S_2'h^{-1}S_2''$ is a zero-sum sequence  with length greater that $S_2'$, yielding a contradiction to  the definition of $S_2'$.

Therefore, $|\Sigma_{\geq k}(S)|\geq  r+|\mathrm{supp}(S)|+1$.

This completes the proof.
\end{proof}

\noindent{\bf Proof of Theorem~\ref{GGXNZSFS}}

\begin{proof}
Note that this theorem is translation-invariant, so we may assume that $\mathsf v_{0}(S)= \mathsf h(S)$. Since $0\not\in \Sigma_{n}(S)$, we infer that $\mathsf v_{0}(S)=\mathsf h(S) \leq n-1$. Let $k=n-\mathsf h(S)$. Then $1\leq k <n$, and we can write $S$ as
\begin{equation*}
S=0^{n-k}T
\end{equation*}
with $0\not\in \mathrm{supp}(T)$. Then we have
\begin{equation*}
|\mathrm{supp}(T)|=|\mathrm{supp}(S)|-1.
\end{equation*}
Since $|S|=n+r >n$, we infer that $|T|=|S|-(n-k)=r +k > k$. By Lemma~\ref{Lemma 2.5}, we obtain that
\begin{equation*}
\Sigma_{n}(S)=\Sigma_{\geq k }(T).
\end{equation*}
Since $0\not\in \Sigma_{n}(S)$, we infer that $0\not\in \Sigma_{\geq k}(T)$. By Theorem~\ref{mainresults3}, we obtain that
\begin{equation*}
|\Sigma_{n}(S)|=|\Sigma_{\geq k }(T)|\geq |T|-k+|\mathrm{supp}(T)|=r+|\mathrm{supp}(S)|-1.
\end{equation*}
We are done.
\end{proof}

\section{Concluding remarks}

By Theorem~\ref{mainresults2}, we obtain a counterexample of Conjecture~\ref{conjecture of GGX}. It is interesting to determine when Conjecture~\ref{conjecture of GGX} holds true.

Let $G$ be a finite abelian group, and $k,r$ two positive integers. Let
\begin{align*}
\mathsf f_{G,k}(r)=\min\{|\Sigma_{k}(S)| \mid & S \mbox{ is a sequence over } G \\ & \mbox{ with } |S|= k+r \mbox{ and } 0\not\in \Sigma_{k}(S)\}.
\end{align*}

Based on some known results, we suggest the following conjecture.

\begin{conjecture}\label{new conjecture}
Let $G$ be a finite abelian group and $S$ a sequence over $G$ with $|S|=|G|+r$, where $r \in [0, \mathsf D(G)-2]$. Suppose  $0\not\in \Sigma_{|G|}(S)$. If $|\Sigma_{|G|}(S)|=\mathsf f_{G,|G|}(r)$, then  there exists a zero-sum free sequence $T$ over $G$ with $|T|=r+1$ such that $|\Sigma_{|G|}(S)| \geq |\Sigma(T)|$ and $|\mathrm{supp}(T)|\geq \min\{r+1, |\mathrm{supp}(S)|-1\}$.
\end{conjecture}

We have the following results

\begin{proposition}
Conjecture~\ref{new conjecture} holds true in the following cases.
\begin{itemize}
\item[1.] $G$ is a finite abelian group and $1\leq r\leq \exp(G)-2$. In this case, we have $\mathsf f_{G,|G|}(r)=r+1$.
\item[2.] $G=C_p \oplus C_p$ and $p-1\leq r \leq 2p-3$. In this case, we have $\mathsf f_{G,|G|}(r)=(r-p+3)p-1$.
\item[3.] $G=C_{n_1} \oplus C_{n_2}$ with $1 < n_1 \mid n_2$  and $r \in \{n_1+n_2-4, n_1+n_2-3\}$. In this case, we have $\mathsf f_{G,|G|}(r)=(r-n_1+3)n_2-1$.
\end{itemize}
\end{proposition}

\begin{proof}
1. The result follows from \cite[Theorem 2]{Yu2004} and independently \cite[Theorem 1.1]{XQQ2014}.

2. The result follows from \cite[Theorem 1.7]{PWHL2023} and \cite[Theorem 1.8]{PWHL2023}.

3. The result follows from \cite[Theorem 1.1]{HHL} and \cite[Theorem 1.2]{HHL}.
\end{proof}

We close this paper by the following problems.

\begin{problem}
Determine the invariant $\mathsf f_{G,k}(r)$, where $G$ is a finite abelian group and $k,r$ are positive integers.
\end{problem}

\begin{problem}
Determine the structure of the sequence $S$ over a finite abelian group $G$ such that the equality holds in Theorem~\ref{mainresults3}.
\end{problem}

\bigskip

\noindent {\bf Acknowledgements}. This research was supported in part by  the National Natural Science Foundation of China (Grant No. 12271520), and was also supported in part by the Fundamental Research Funds for the Central Universities of Civil Aviation University of China (Grant No. 3122022082).

\end{document}